\documentclass[11pt,reqno,twoside]{amsart}
\usepackage{amssymb,amsmath,amsthm,soul,color,paralist}
\usepackage{t1enc}
\usepackage{comment}
\usepackage[cp1250]{inputenc}
\usepackage{a4,indentfirst,latexsym}
\usepackage{graphics}
\usepackage{mathrsfs}
\usepackage{cite,enumitem,graphicx}
\usepackage[colorlinks=true,urlcolor=blue,
citecolor=red,linkcolor=blue,linktocpage,pdfpagelabels,
bookmarksnumbered,bookmarksopen]{hyperref}
\usepackage[english]{babel}
\usepackage[top=2.50cm, bottom=2.50cm, left=2.60cm, right=2.60cm]{geometry}
\usepackage[metapost]{mfpic}
\usepackage[colorinlistoftodos]{todonotes}
\usepackage[normalem]{ulem}
\usepackage{comment}

\makeatletter
\providecommand\@dotsep{5}
\def\listtodoname{List of Todos}
\def\listoftodos{\@starttoc{tdo}\listtodoname}
\makeatother

\numberwithin{equation}{section}

\newcommand{\R}{\mathbb{R}}
\newcommand{\C}{\mathcal{C}}
\newcommand{\dis}{\displaystyle}

\DeclareMathOperator{\dive}{div}
\DeclareMathOperator{\supp}{supp}
\DeclareMathOperator{\B}{\mathcal{B}}
\DeclareMathOperator{\I}{\mathcal{I}}
\DeclareMathOperator{\E}{\mathbb{E}}

\DeclareMathOperator{\e}{\varepsilon}
\DeclareMathOperator{\la}{\lambda}
\DeclareMathOperator{\ri}{\rightarrow}
\DeclareMathOperator{\rih}{\rightharpoonup}

\newtheorem{prop}{Proposition}[section]
\newtheorem{lem}{Lemma}[section]
\newtheorem{thm}{Theorem}[section]

\newtheorem{remark}{Remark}[section]

\begin{document}
\title[Double phase Kirchhoff problems with vanishing potentials]{Nodal solutions for double phase Kirchhoff \\
problems with vanishing potentials}

\author[T. Isernia]{Teresa Isernia}
\address{Teresa Isernia\hfill\break\indent
Dipartimento di Ingegneria Industriale e Scienze Matematiche \hfill\break\indent
Universit\`a Politecnica delle Marche\hfill\break\indent
Via Brecce Bianche, 12\hfill\break\indent
60131 Ancona (Italy)}
\email{t.isernia@univpm.it}

\author[D.D. Repov\v{s}]{Du\v{s}an D. Repov\v{s}}
\address{Du\v{s}an D. Repov\v{s}\hfill\break\indent
Faculty of Education, and Faculty of Mathematics and Physics \hfill\break\indent
University of Ljubljana\hfill\break\indent
\& Institute of Mathematics, Physics and Mechanics \hfill\break\indent
SI-1000 Ljubljana, Slovenia}
\email{dusan.repovs@guest.arnes.si}



\begin{abstract}
We consider the following $(p, q)$-Laplacian Kirchhoff type problem 
\begin{align*}
\begin{split}
&-\left(a+b\int_{\R^{3}}|\nabla u|^{p}\, dx \right)\Delta_{p}u - \left(c+d\int_{\R^{3}}|\nabla u|^{q}\, dx \right ) \Delta_{q}u \\
&\qquad + V(x) (|u|^{p-2}u + |u|^{q-2}u)= K(x) f(u) \quad \mbox{ in } \R^{3},
\end{split}
\end{align*}
where $a, b, c, d>0$ are constants, $\frac{3}{2}< p< q<3$, $V: \R^{3}\ri \R$ and $K: \R^{3}\ri \R$ are positive continuous functions allowed for vanishing behavior at infinity, and $f$ is a continuous function with quasicritical growth. Using a minimization argument and a quantitative deformation lemma we establish the existence of nodal solutions.
\end{abstract}

\maketitle

\section{Introduction}

\noindent
This paper deals with the existence of least energy nodal solutions for the following class of quasilinear problems
\begin{align}\begin{split}\label{P}
&-\left(a+b\int_{\R^{3}}|\nabla u|^{p}\, dx \right)\Delta_{p}u - \left(c+d\int_{\R^{3}}|\nabla u|^{q}\, dx \right ) \Delta_{q}u \\
&\qquad + V(x) (|u|^{p-2}u + |u|^{q-2}u)= K(x) f(u) \quad \mbox{ in } \R^{3},
\end{split}\end{align}
where $a, b, c, d>0$ are constants, $\frac{3}{2}< p< q<3$, $V: \R^{3}\ri \R$ and $K: \R^{3}\ri \R$ are positive functions,
 and $f$ is a continuous function with quasicritical growth. 

In recent years, a considerable interest has been devoted to the study of this general class of problems due to the fact that they arise in applications in physics and related sciences. 

When $a=c=1$ and $b=d=0$, equation \eqref{P} becomes a $(p, q)$-Laplacian problem of the type
\begin{equation}\label{pq}
-\Delta_{p}u-\Delta_{q} u+V(x) \left(|u|^{p-2}u+|u|^{q-2}u\right)= K(x)f(u) \quad \mbox{ in } \R^{3}. 
\end{equation}
As underlined in \cite{CIL}, this equation is related to the more general reaction-diffusion system
$$
u_{t}=\dive(D(u)\nabla u)+c(x,u) \quad \mbox{ and } \quad D(u)=|\nabla u|^{p-2}+|\nabla u|^{q-2},
$$
which appears in plasma physics, biophysics and chemical reaction design. 
 
In these applications, $u$ represents a concentration, $\dive(D(u) \nabla u)$ is the diffusion with the diffusion coefficient $D(u)$, and the reaction term $c(x, u)$ relates to source and loss processes. 
Usually, the reaction term $c(x, u)$ is a polynomial of $u$ with variable coefficient (see \cite{CIL}). This kind of problem has been widely investigated by many authors, see for instance \cite{CIL, Fig1, Fig2, HL, KR, LG, LL, MP} and the references therein. 
In particular, in \cite{BFJMAA}, using a minimization argument and a quantitative deformation lemma, the authors proved the existence of nodal solutions for the following class of $(p, q)$ problems 
\begin{align*}
-\dive(a(|\nabla u|^{p}) |\nabla u|^{p-2} \nabla u) + V(x) b(|u|^{p})|u|^{p-2}u= K(x)f(u) \quad \mbox{ in } \R^{N}, 
\end{align*}
where $N\geq 3$, $2\leq p<N$, $a, b, f \in \C^{1}(\R)$,
 and $V, K$ are continuous and positive functions (see also \cite{BFANS}).   

We stress that in the nonlocal framework, only few recent works deal with the fractional $(p, q)$-Laplacian. In \cite{ChenBao} the authors established the existence, nonexistence and multiplicity for a nonlocal $(p, q)$-subcritical problem. Ambrosio \cite{AZAA} obtained an existence result for a critical fractional $(p, q)$-problem via mountain pass theorem. In \cite{BM} the authors investigated the existence of infinitely many nontrivial solutions for a class of fractional $(p, q)$-equations involving concave-critical nonlinearities in bounded domains. H\"older regularity result for nonlocal double phase equations has been established in \cite{DFP}. Applying suitable variational and topological arguments, in \cite{AR} the authors obtained a multiplicity and concentration result for a class of fractional problems with unbalanced growth. We also mention \cite{AAI, AIMed, GKS} for other interesting results.

We underline that there is a huge bibliography concerning the nonlinear Schr\"odinger equation (that is when $p=q=2$ in \eqref{pq})
\begin{align}\label{schrodinger}
-\Delta u + V(x) u= K(x) f(u) \quad \mbox{ in } \R^{3},
\end{align}
and we would like to point out that an important class of problems associated with \eqref{schrodinger} is the so called {\it zero mass} case, which occurs when the potential $V$ vanishes at infinity. Using several variational methods, many authors attacked this equation; see for instance \cite{AS, AS1, AFM, AW, BGM, BerLions, BVS}. 
\smallskip

When $a=c$, $b=d(\neq 0)$ and $p=q=2$, problem \eqref{P} becomes the following Kirchhoff equation
\begin{align}\label{Ke}
-\left( a+b \int_{\R^{3}} |\nabla u|^{2} \, dx\right) \Delta u + V(x) u = f(x, u) \quad \mbox{ in } \R^{3}.  
\end{align}
This problem is related to the stationary analogue of the Kirchhoff equation \cite{K} 
\begin{align*}
\rho \frac{\partial^{2}u}{\partial t^{2}} - \left( \frac{p_{0}}{h} + \frac{E}{2L} \int_{0}^{L} \left|\frac{\partial u}{\partial x}\right|^{2} dx \right) \frac{\partial^{2}u}{\partial x^{2}} =0, 
\end{align*}
for all $x\in (0, L)$ and $t\geq 0$. This equation is an extension of the classical D'Alembert wave equation taking into account the changes in the length of the strings produced by transverse vibrations. In \eqref{Ke}, $u(x, t)$ is the lateral displacement of the vibrating string at the coordinate $x$ and the time $t$, $L$ is the length of the string, $h$ is the cross-section area, $E$ is the Young modulus of the material, $\rho$ is the mass density and $p_{0}$ is the initial axial tension. 

The early studies dedicated to the Kirchhoff equation \eqref{Ke} were done by Bernstein \cite{Bernstein} and Pohozaev \cite{Pohozaev}. However, the Kirchhoff equation \eqref{Ke} began to attract the attention of more researchers only after the work by Lions \cite{LionsK}, in which the author introduced a functional analysis approach to study a general Kirchhoff equation in arbitrary dimension with external force term. For more details on classical Kirchhoff problems we refer to \cite{AP, APS, PZ1, PZ}.
In \cite{FJ} the authors established the existence of a least energy nodal solution to the following class of nonlocal Schr\"odinger-Kirchhoff problems
\begin{align*}
M\left( \int_{\R^{3}} |\nabla u|^{2}\, dx + \int_{\R^{3}} V(x) u^{2}\, dx \right) (-\Delta u + V(x) u)= K(x) f(u) \quad \mbox{ in } \R^{3}. 
\end{align*}
Moreover, when the problem presents symmetry, they proved the existence of infinitely many nontrivial solutions. We also mention \cite{DPS, FN} where the existence of nodal solutions for problems like \eqref{Ke} has been obtained.

In the nonlocal framework, Fiscella and Valdinoci \cite{FisVal} proposed the following stationary Kirchhoff model driven by the fractional Laplacian 
\begin{align}\label{FV}
\left\{
\begin{array}{ll}
\displaystyle{-M\left( \iint_{\R^{2N}} \frac{|u(x)- u(y)|^{2}}{|x-y|^{N+2s}}\, dxdy\right) (-\Delta)^{s} u= \la f(x, u) + |u|^{2^{*}_{s}-2} u} &\mbox{ in } \Omega, \\
u=0 &\mbox{ in } \R^{N}\setminus \Omega, 
\end{array}
\right. 
\end{align}
where $\Omega \subset \R^{N}$ is an open bounded set, $2^{*}_{s}= \frac{2N}{N-2s}$, $N>2s$, $s\in (0,1)$, $M:\R^{+}\ri \R^{+}$ is an increasing continuous function which behaves like $M(t)= a+b \,t$, with $b\geq 0$, and {$f$ is a continuous function}. Based on a truncation argument and the mountain pass theorem, the authors established the existence of a non-negative solution to \eqref{FV} for any $\la>\la^{*}>0$, where $\la^{*}$ is an appropriate threshold. 
We also mention \cite{AAsy, AICCM, AIMM, AFP, MRZ, PucciS} in which the authors dealt with existence and multiplicity of solutions for \eqref{FV}, while concerning the existence and multiplicity of sign-changing solutions for fractional Kirchhoff problems only few results appear in the literature \cite{CG, CTL, INA}.  

\smallskip
Finally, if $a=c$, $b=d(\neq 0)$ and $p=q\neq 2$,
 we have the following $p$-Laplacian Kirchhoff-type equation
\begin{align}\label{Kep}
-\left( a+b \int_{\R^{3}} |\nabla u|^{p} \, dx\right) \Delta_{p} u + V(x) |u|^{p-2}u = f(x, u) \quad \mbox{ in } \R^{3}.  
\end{align}
Very recently, in \cite{HMH} using a minimization argument and the Nehari manifold method, the authors investigated the existence of least energy nodal (or sign-changing) solutions to \eqref{Kep}. We also mention \cite{CP, CV, GN, X} for results regarding Schr\"odinger-Kirchhoff equations involving the $p$-Laplacian. 

Motivated by the interest shared by the mathematical community toward $(p, q)$-Laplacian problems, the goal of the present paper is to study the existence of nodal solutions to \eqref{P}. In order to state precisely our main result, we first introduce the main assumptions on the potentials $V$ and $K$ and on the nonlinearity $f$.  

We assume that $V, K: \R^{3}\ri \R$ are continuous functions and we say that $(V, K)\in \mathcal{K}$ if the following conditions are satisfied (see \cite{AS}):

\begin{compactenum}
\item [$(VK_{1})$] $V(x), K(x)>0$ for all $x\in \R^{3}$ and $K\in L^{\infty}(\R^{3})$; 
\item [$(VK_{2})$] If $(\mathcal{A}_{n})\subset \R^{3}$ is a sequence of Borel sets such that the Lebesgue measure $|\mathcal{A}_{n}|\leq R$, for all $n\in \mathbb{N}$ and for some $R>0$, then
\begin{align*}
\lim_{r\ri \infty} \int_{\mathcal{A}_{n}\cap \B_{\varrho}^{c}(0)} K(x)\, dx =0, 
\end{align*}  
uniformly in $n\in \mathbb{N}$, where $\B_{\varrho}^{c}(0):=\R^{3}\setminus \B_{\varrho}(0)$. 
\end{compactenum}
Furthermore, one of the following conditions is satisfied:
\begin{compactenum}
\item [$(VK_{3})$] $\dis \frac{K}{V}\in L^{\infty}(\R^{3})$;
\end{compactenum}
or
\begin{compactenum}
\item [$(VK_{4})$] There exists $m\in (q, q^{*})$ such that 
\begin{align*}
\frac{K(x)}{V(x)^{\frac{q^{*}-m}{q^{*}-p}}}\ri 0 \quad \mbox{ as } |x|\ri \infty. 
\end{align*}
\end{compactenum}

Let us point out that the hypotheses on the functions $V$ and $K$ characterize problem \eqref{P} as a {\it zero mass} problem. 

Regarding the nonlinearity $f$, we assume that $f\in \C(\R, \R)$ and $f$ fulfills the following conditions: 
\begin{compactenum}
\item [$(f_{1})$] $\dis \lim_{|t|\ri 0} \frac{f(t)}{|t|^{2p-1}} =0$ if $(VK_{3})$ holds, 
\item [$(\tilde{f}_{1})$] $\dis \lim_{|t|\ri 0} \frac{f(t)}{|t|^{m-1}} =0$ if $(VK_{4})$ holds, with $m\in (q, q^{*})$ defined in $(VK_{4})$, 
\item [$(f_{2})$] $\dis \lim_{|t|\ri \infty} \frac{f(t)}{|t|^{q^{*}-1}}=0$,
\item [$(f_{3})$] $\dis \lim_{t\ri \infty} \frac{F(t)}{t^{2q}}=\infty$, where $\dis F(t):=\int_{0}^{t} f(\tau)\, d\tau$,  
\item [$(f_{4})$] The map $\dis t\mapsto \frac{f(t)}{|t|^{2q-1}}$ is strictly increasing for all $|t|>0$. 
\end{compactenum}

We note that from assumption $(f_{4})$ it follows that $t\mapsto \frac{1}{2q} f(t)t-F(t)$ is increasing for $t\geq 0$ and also that $t\mapsto \frac{1}{2q} f(t)t-F(t)$ is decreasing for $t\leq 0$ (see Remark \ref{rem1} below). 
\medskip

Our main result can be stated as follows: 
\begin{thm}\label{thm1}
Assume that $(V, K)\in \mathcal{K}$ and $f$ satisfies conditions $(f_{1})$ $($or $(\tilde{f}_{1}))$ and $(f_{2})$-$(f_{4})$. Then problem \eqref{P} admits a least energy sign-changing weak solution. If in addition, $f$ is an odd function, then \eqref{P} has infinitely many nontrivial solutions. 
\end{thm}

A {\it weak solution} of problem \eqref{P} is a function $u\in \E$ such that 
\begin{align}\begin{split}\label{ws}
&a\int_{\R^{3}} |\nabla u|^{p-2} \nabla u \cdot \nabla \varphi \, dx + b \left(\int_{\R^{3}} |\nabla u|^{p}\, dx \right) \int_{\R^{3}} |\nabla u|^{p-2} \nabla u \cdot \nabla \varphi\, dx\\
&\quad +c\int_{\R^{3}} |\nabla u|^{q-2} \nabla u \cdot \nabla \varphi \, dx + d \left(\int_{\R^{3}} |\nabla u|^{q}\, dx \right) \int_{\R^{3}} |\nabla u|^{q-2} \nabla u \cdot \nabla \varphi\, dx\\
&\quad +\int_{\R^{3}} V(x)(|u|^{p-2} u\varphi + |u|^{q-2}u\varphi)\, dx - \int_{\R^{3}} K(x)f(u) \varphi\, dx=0
\end{split}\end{align}
for all $\varphi \in \E$, where 
\begin{align*}
\E= \left\{u\in \mathcal{D}^{1, p}(\R^{3})\cap \mathcal{D}^{1, q}(\R^{3}) \, : \, \int_{\R^{3}} V(x) (|u|^{p}+ |u|^{q})\, dx <\infty \right\}. 
\end{align*}
By a sign-changing weak solution to problem \eqref{P} we mean a function $u\in \E$ that satisfies \eqref{ws} with $u^{+}=\max\{u, 0\}\neq 0$ and $u^{-}=\min\{u, 0\}\neq 0$. 

The proof of Theorem \ref{thm1} is achieved by using suitable variational techniques inspired by \cite{AS2, BFJMAA, BFANS, FJ}. In order to study \eqref{P} we consider the following functional $\I:\E\ri \R$ given by 
\begin{align*}
\I(u)&= \frac{a}{p} \int_{\R^{3}} |\nabla u|^{p}\, dx + \frac{b}{2p} \left( \int_{\R^{3}} |\nabla u|^{p}\, dx \right)^{2} + \frac{c}{q} \int_{\R^{3}} |\nabla u|^{q}\, dx + \frac{d}{2q} \left( \int_{\R^{3}} |\nabla u|^{q}\, dx \right)^{2}\\
&\quad +\int_{\R^{3}} V(x)\left(\frac{1}{p}|u|^{p} + \frac{1}{q}|u|^{q}\right)\, dx - \int_{\R^{3}} K(x)F(u)\, dx.
\end{align*} 
It is easy to check that $\I\in \C^{1}(\E, \R)$ and its differential is given by 
\begin{align*}
\langle \I'(u), \varphi \rangle &=  a\int_{\R^{3}} |\nabla u|^{p-2} \nabla u \cdot \nabla \varphi \, dx + b \left(\int_{\R^{3}} |\nabla u|^{p}\, dx \right) \int_{\R^{3}} |\nabla u|^{p-2} \nabla u \cdot \nabla \varphi\, dx\\
&\quad +c\int_{\R^{3}} |\nabla u|^{q-2} \nabla u \cdot \nabla \varphi \, dx + d \left(\int_{\R^{3}} |\nabla u|^{q}\, dx \right) \int_{\R^{3}} |\nabla u|^{q-2} \nabla u \cdot \nabla \varphi\, dx\\
&\quad +\int_{\R^{3}} V(x)(|u|^{p-2} u\varphi + |u|^{q-2}u\varphi)\, dx - \int_{\R^{3}} K(x)f(u) \varphi\, dx.
\end{align*}
Then, we define the nodal set 
\begin{equation*}
\mathcal{M}= \left\{ w \in \mathcal{N}\, : \, w^{\pm}\neq 0, \, \langle \I'(w), w^{+}\rangle = \langle \I'(w), w^{-}\rangle=0\right\},  
\end{equation*}
where 
\begin{equation*}
\mathcal{N}= \left\{ u\in \E\setminus \{0\} \, : \, \langle \I'(u), u \rangle=0 \right\}.  
\end{equation*} 
In order to get least energy nodal (or sign-changing) solutions to \eqref{P}, we minimize the functional $\I$ on the nodal set $\mathcal{M}$. Then we prove that the minimum is achieved and, by using a variant of the quantitative deformation lemma, we show that it is a critical point of $\I$. Finally, when the nonlinearity $f$ is odd, we obtain the existence of infinitely many nontrivial weak solutions not necessarily nodals.
We point out that our paper extends the results obtained in \cite{FJ, HMH}. 

Problem \eqref{P} is called nonlocal due to the presence of the Kirchhoff term $\left(\int_{\R^{3}} |\nabla u|^{p}\, dx\right)\Delta_{p}u$, this causes some mathematical difficulties which makes the study of such a class of problems particularly interesting. We underline that here we are considering the sum of two Kirchhoff terms: $\left(\int_{\R^{3}} |\nabla u|^{p}\, dx\right)\Delta_{p}u$ and $\left(\int_{\R^{3}} |\nabla u|^{q}\, dx\right)\Delta_{q}u$, with $p<q$. 
Moreover, due to the fact that the nonlinearity $f$ is only continuous, one cannot apply standard $\C^{1}$-Nehari manifold arguments due to the lack of differentiability of the associated Nehari manifold $\mathcal{N}$. We were able to overcome this difficulty 
by
borrowing some abstract critical point results obtained in \cite{SW}. Furthermore, to produce nodal solutions, instead of using the Miranda Theorem to get critical points of $g_{u}(\xi, \la)= \I(\xi u^{+}+\la u^{-})$ we use an iterative process to build a sequence which converges to a critical point of $g_{u}(\xi, \la)$.

The paper is organized as follows. In Section \ref{sect2} we introduce the variational structure. In Section \ref{sect3} we give some preliminary results which overcome the lack of differentiability of the Nehari manifold. Section \ref{sect4} is devoted to some technical lemmas used in the proof of the main result. In Section \ref{sect5} we prove Theorem \ref{thm1}. 

\subsection{Notations}
We denote by $\B_{R}(x)$ the ball of radius $R$ with center $x$ and we set $\B_{R}^{c}(x)=\R^{3}\setminus \B_{R}(x)$.  
Let $1\leq r\leq \infty$ and $A\subset \R^{3}$. We denote by $|u|_{L^{r}(A)}$ the $L^{r}(A)$-norm of the function $u:\R^{3}\ri \R$ belonging to $L^{r}(A)$. When $A=\R^{3}$, we 
shall
 simply write $|u|_{r}$. 

\section{Variational framework}\label{sect2}

\noindent
Let us introduce the space
\begin{align*}
\E= \left\{u\in \mathcal{D}^{1, p}(\R^{3})\cap \mathcal{D}^{1, q}(\R^{3}) \, : \, \int_{\R^{3}} V(x) (|u|^{p}+ |u|^{q})\, dx <\infty \right\}
\end{align*}
endowed with the norm 
\begin{align*}
\|u\|= \left(\int_{\R^{3}} \left(a|\nabla u|^{p}+ V(x)|u|^{p}\right)\, dx\right)^{\frac{1}{p}} + \left(\int_{\R^{3}} \left(c|\nabla u|^{q}+ V(x)|u|^{q}\right)\, dx\right)^{\frac{1}{q}}. 
\end{align*}

Let us define the Lebesgue space
\begin{align*}
L^{r}_{K}(\R^{3})= \left\{ u:\R^{3}\ri \R \, : \, u \mbox{ is measurable and } \int_{\R^{3}} K(x) |u|^{r}\, dx <\infty \right\}
\end{align*}
equipped with the norm
\begin{equation*}
\|u\|_{L^{r}_{K}(\R^{3})}= \left( \int_{\R^{3}} K(x)|u|^{r}\, dx \right)^{\frac{1}{r}}. 
\end{equation*}

We recall the following continuous and compactness results whose proofs can be found in \cite{BFJMAA}:

\begin{lem}\label{lem1}
Assume that $(V, K)\in \mathcal{K}$. 
\begin{compactenum}[$(i)$]
\item If $(VK_{3})$ holds, then $\E$ is continuously embedded in $L^{r}_{K}(\R^{3})$ for every $r\in [q, q^{*}]$. 
\item If $(VK_{4})$ holds, then $\E$ is continuously embedded in $L^{m}_{K}(\R^{3})$.
\end{compactenum}
\end{lem}

\begin{lem}\label{lem2}
Assume that $(V, K)\in \mathcal{K}$. 
\begin{compactenum}[$(i)$]
\item If $(VK_{3})$ holds, then $\E$ is compactly embedded in $L^{r}_{K}(\R^{3})$ for every $r\in (q, q^{*})$. 
\item If $(VK_{4})$ holds, then $\E$ is compactly embedded in $L^{m}_{K}(\R^{3})$.
\end{compactenum}
\end{lem}

The last lemma of this section is a compactness result related to the nonlinearity (see \cite{BFJMAA}).

\begin{lem}\label{lemconvergFf} 
Assume that $(V, K)\in \mathcal{K}$ and $f$ satisfies $(f_{1})$-$(f_{2})$ or $(\tilde{f}_{1})$-$(f_{2})$. If $(u_{n})$ is a sequence such that $u_{n}\rih u$ in $\E$, then
\begin{align*}
\int_{\R^{3}} K(x)F(u_{n})\, dx \ri \int_{\R^{3}} K(x)F(u)\, dx 
\end{align*}
and 
\begin{align*}
\int_{\R^{3}} K(x)f(u_{n})u_{n}\, dx \ri \int_{\R^{3}} K(x)f(u)u\, dx.
\end{align*}
\end{lem}

We conclude this section
by
 giving the following useful remarks. 
\begin{remark}\label{rem1}
Let us point out that from assumption $(f_{4})$ it follows that $t\mapsto \frac{1}{2q} f(t)t-F(t)$ is increasing for $t\geq 0$. Indeed, let $0<t_{2}<t_{1}$, then using $(f_{4})$ twice we get
\begin{align*}
\frac{1}{2q}f(t_{1})t_{1}- F(t_{1})&= \frac{1}{2q} f(t_{1}) t_{1}- F(t_{2})- \int_{t_{2}}^{t_{1}} f(\tau)\, d\tau \\
&=\frac{1}{2q} f(t_{1}) t_{1}- F(t_{2}) - \int_{t_{2}}^{t_{1}} \frac{f(\tau)}{\tau^{2q-1}} \, \tau^{2q-1}\, d\tau \\
&>\frac{1}{2q} f(t_{1}) t_{1}- F(t_{2}) -  \frac{f(t_{1})}{t_{1}^{2q-1}} \int_{t_{2}}^{t_{1}} \tau^{2q-1}\, d\tau \\
&=\frac{1}{2q} f(t_{1}) t_{1}- F(t_{2}) -  \frac{f(t_{1})}{t_{1}^{2q-1}} \frac{t_{1}^{2q}-t_{2}^{2q}}{2q} \\
&=\frac{1}{2q} \frac{f(t_{1})}{t_{1}^{2q-1}} t_{2}^{2q}- F(t_{2}) \\
&>\frac{1}{2q} f(t_{2})t_{2}-F(t_{2}). 
\end{align*}
Similarly, it is possible to prove that $t\mapsto \frac{1}{2q} f(t)t-F(t)$ is decreasing for $t\leq 0$. 
\end{remark}

\begin{remark}
Take $u\in \E$ with $u^{\pm}\neq 0$ and $\xi, \la \geq 0$, then
\begin{align*}
|\nabla (\xi u^{+}+ \la u^{-})|^{p}= |\nabla (\xi u^{+})|^{p}+ |\nabla (\la u^{-})|^{p}
\end{align*}
and using the linearity of $F$ and the positivity of $K$ we also have 
\begin{align*}
\int_{\R^{3}} K(x) F(\xi u^{+}+ \la u^{-})\, dx = \int_{\R^{3}} K(x) \left(\, F(\xi u^{+}) + F(\la u^{-})\, \right)\, dx. 
\end{align*}
Hence, for any $u\in \E$ with $u^{\pm}\neq 0$ and $\xi, \la \geq 0$ we have
\begin{align}\label{somma}
\I(\xi u^{+}+ \la u^{-})= \I(\xi u^{+})+ \I(\la u^{-}).
\end{align}
Moreover, 
\begin{align*}
\langle \I'(\xi u^{+}+ \la u^{-}), \xi u^{+}\rangle &= \xi^{p} \int_{\R^{3}} \left(a|\nabla u^{+}|^{p}+ V(x)|u|^{p}\right)\, dx + b\xi^{2p} \left(\int_{\R^{3}}|\nabla u^{+}|^{p}\, dx\right)^{2}\\
&\quad+\xi^{q} \int_{\R^{3}} \left(c|\nabla u^{+}|^{q}+ V(x)|u|^{q}\right)\, dx + d\xi^{2q} \left(\int_{\R^{3}}|\nabla u^{+}|^{q}\, dx\right)^{2}\\
&\quad-\int_{\R^{3}} K(x) f(\xi u^{+}) \,\xi u^{+}\, dx  
\end{align*}
and 
\begin{align*}
\langle \I'(\xi u^{+}+ \la u^{-}), \xi u^{+}\rangle &= \la^{p} \int_{\R^{3}} \left(a|\nabla u^{-}|^{p}+ V(x)|u|^{p}\right)\, dx + b\la^{2p} \left(\int_{\R^{3}}|\nabla u^{-}|^{p}\, dx\right)^{2}\\
&\quad+\la^{q} \int_{\R^{3}} \left(c|\nabla u^{-}|^{q}+ V(x)|u|^{q}\right)\, dx + d\la^{2q} \left(\int_{\R^{3}}|\nabla u^{-}|^{q}\, dx\right)^{2}\\
&\quad-\int_{\R^{3}} K(x) f(\la u^{-}) \,\la u^{-}\, dx.  
\end{align*}
\end{remark}

\section{Preliminaries} \label{sect3}

\noindent
The Nehari manifold associated with $\I$ is given by 
\begin{equation*}
\mathcal{N}= \left\{ u\in \E\setminus \{0\} \, : \, \langle \I'(u), u \rangle=0 \right\}.  
\end{equation*}
We denote by 
\begin{equation*}
\mathcal{M}= \left\{ w \in \mathcal{N}\, : \, w^{\pm}\neq 0, \, \langle \I'(w), w^{+}\rangle = \langle \I'(w), w^{-}\rangle=0\right\},  
\end{equation*}
and by $\mathbb{S}$ the unit sphere on $\E$. We note that $\mathcal{M}\subset \mathcal{N}$.

Once $f$ is only continuous, the following results are crucial, since they allow us to overcome the non-differentiability of $\mathcal{N}$. 

\begin{lem}\label{lemz1}
Suppose that $(V, K)\in \mathcal{K}$ and $f$ satisfies conditions $(f_1)-(f_4)$. Then the following properties hold:
\begin{compactenum}
\item[$(a)$] For each $u\in \E\setminus\{0\}$, let $\varphi_{u}: \R_{+} \rightarrow \R$ be defined by 
$$
\varphi_{u}(t)= \I(tu).
$$ 
Then there is a unique $t_{u}>0$ such that
\begin{align*}
\varphi_{u}'(t)>0 \mbox{ for } t\in (0, t_{u}) \quad \mbox{ and } \quad \varphi_{u}'(t)<0 \mbox{ for } t\in (t_{u}, \infty);  
\end{align*}
\item[$(b)$] There is $\tau>0$, independent of $u$, such that $t_{u}\geq \tau$ for every $u\in \mathbb{S}$. Moreover, for each compact set $\mathbb{K}\subset \mathbb{S}$, there is $C_{\mathbb{K}}>0$ such that $t_{u}\leq C_{\mathbb{K}}$ for every $u\in \mathbb{K}$;
\item[$(c)$] The map $\hat{m}: \E\setminus \{0\}\rightarrow \mathcal{N}$ given by $\hat{m}(u):=t_{u}u$ is continuous and {$m:= \hat{m}|_{\mathbb{S}}$} is a homeomorphism between $\mathbb{S}$ and $\mathcal{N}$. Moreover, $m^{-1}(u)= \frac{u}{\|u\|}$.
\end{compactenum}
\end{lem}

\begin{proof}

\noindent
$(a)$ Let us assume that $(VK_{3})$ holds. Then, using assumptions $(f_{1})$-$(f_{2})$ given $\e>0$ there exists $C_{\e}>0$ such that
\begin{align*}
|f(t)|\leq \e |t|^{p-1} + C_{\e}|t|^{q^{*}-1}.  
\end{align*}
Therefore 
\begin{align*}
\I(tu)&\geq \frac{a}{p}t^{p}\int_{\R^{3}}|\nabla u|^{p}\, dx + \frac{c}{q}t^{q}\int_{\R^{3}}|\nabla u|^{q}\, dx +\int_{\R^{3}} V(x)\left( \frac{t^{p}}{p}|u|^{p}+ \frac{t^{q}}{q}|u|^{q}\right)\, dx \\
&\qquad- \e t^{p} \int_{\R^{3}} K(x) |u|^{p}\, dx - C_{\e} t^{q^{*}} \int_{\R^{3}} K(x) |u|^{q^{*}}\, dx \\
&\geq \frac{a}{p}t^{p}\int_{\R^{3}}|\nabla u|^{p}\, dx + \frac{c}{q}t^{q}\int_{\R^{3}}|\nabla u|^{q}\, dx +\int_{\R^{3}} V(x)\left( \frac{t^{p}}{p}|u|^{p}+ \frac{t^{q}}{q}|u|^{q}\right)\, dx \\
&\qquad-\e \left|\frac{K}{V}\right|_{\infty} t^{p} \int_{\R^{3}} V(x) |u|^{p}\, dx - C'_{\e} |K|_{\infty} t^{q^{*}} \|u\|^{q^{*}}\\
&\geq \frac{a}{p}t^{p}\int_{\R^{3}}|\nabla u|^{p}\, dx + \left(\frac{1}{p}- \e \left|\frac{K}{V}\right|_{\infty}\right) t^{p} \int_{\R^{3}} V(x)|u|^{p}\, dx + \frac{t^{q}}{q}\int_{\R^{3}} \left(c|\nabla u|^{q} + V(x)|u|^{q}\right)\, dx\\
&\qquad- C'_{\e} |K|_{\infty} t^{q^{*}} \|u\|^{q^{*}}.
\end{align*}
Choosing $\e \in \left(0, (2p\left|\frac{K}{V}\right|_{\infty})^{-1}\right)$, we get $t_{0}>0$ sufficiently small such that
\begin{align*}
0<\varphi_{u}(t)= \I(tu) \quad \ \  \mbox{for all} \ \  t\in (0, t_{0}).
\end{align*}
Now, we assume that $(VK_{4})$ is true. Then, there exists a positive constant $C_{m}$ such that, for each $\e \in (0, C_{m})$ we get $R>0$ such that for any $u\in \E$
\begin{align*}
\int_{\B_{R}^{c}(0)} K(x) |u|^{m}\, dx \leq \e \int_{\B_{R}^{c}(0)}(V(x) |u|^{p} + |u|^{q^{*}})\, dx. 
\end{align*}
Using assumptions $(\tilde{f}_{1})$ and $(f_{2})$, the H\"older and Sobolev inequality we get
\begin{align*}
\I(tu)&\geq \frac{a}{p}t^{p}\int_{\R^{3}}|\nabla u|^{p}\, dx + \frac{c}{q}t^{q}\int_{\R^{3}}|\nabla u|^{q}\, dx +\int_{\R^{3}} V(x)\left( \frac{t^{p}}{p}|u|^{p}+ \frac{t^{q}}{q}|u|^{q}\right)\, dx \\
&\qquad- C_{1} t^{m} \left(\int_{\B_{R}(0)} K(x) |u|^{m}\, dx+\int_{\B_{R}^{c}(0)} K(x) |u|^{m}\, dx\right) - C_{2} t^{q^{*}} \int_{\R^{3}} K(x) |u|^{q^{*}}\, dx \\
&\geq \frac{a}{p}t^{p}\int_{\R^{3}}|\nabla u|^{p}\, dx + \frac{c}{q}t^{q}\int_{\R^{3}}|\nabla u|^{q}\, dx +\int_{\R^{3}} V(x)\left( \frac{t^{p}}{p}|u|^{p}+ \frac{t^{q}}{q}|u|^{q}\right)\, dx \\
&\qquad - C_{1} t^{m} |K|_{\frac{q^{*}}{q^{*}-m}} |u|_{q^{*}}^{m} -C_{1} t^{m} \e \int_{\B_{R}^{c}(0)}(V(x) |u|^{p} + |u|^{q^{*}})\, dx - C_{2}'t^{q^{*}} |K|_{\infty} \|u\|^{q^{*}}\\
&\geq \frac{a}{p}t^{p}\int_{\R^{3}}|\nabla u|^{p}\, dx + \frac{c}{q}t^{q}\int_{\R^{3}}|\nabla u|^{q}\, dx +\int_{\R^{3}} V(x)\left( \frac{t^{p}}{p}|u|^{p}+ \frac{t^{q}}{q}|u|^{q}\right)\, dx \\
&\qquad - C_{2}'t^{q^{*}} |K|_{\infty} \|u\|^{q^{*}} - \tilde{C} t^{m} \left[ |K|_{\frac{q^{*}}{q^{*}-m}} |u|_{q^{*}}^{m} + \e \|u\|^{p} + \e \|u\|^{q^{*}}\right]. 
\end{align*}

Therefore there exists $t_{0}>0$ sufficiently small such that
\begin{align*}
0<\varphi_{u}(t)= \I(tu) \quad  \ \  \mbox{for all} \ \ t\in (0, t_{0}).
\end{align*}
Let $A\subset \supp u$ be a measurable set with finite and positive measure. From $F(t)\geq 0$ for any $t\in \R$, $1<p<q$, and combining assumptions $(f_{3})$ together with Fatou's lemma, we obtain 
\begin{align*}
&\limsup_{t\ri \infty} \frac{\I(tu)}{\|tu\|^{2q}}\\
&\leq \limsup_{t\ri \infty} \left\{\frac{1}{p} \frac{1}{t^{2q-p}\|u\|^{2q-p}} + \frac{b}{2p} \frac{1}{t^{2(q-p)}\|u\|^{2(q-p)}} +\frac{1}{q} \frac{1}{t^{q}\|u\|^{q}}+ \frac{d}{2q}- \int_{A} K(x) \frac{F(tu)}{(tu)^{2q}} \, \left(\frac{u}{\|u\|}\right)^{2q} \, dx	\right\}\\
&\leq \frac{d}{2q} - \liminf_{t\ri \infty} \int_{A} K(x) \frac{F(tu)}{(tu)^{2q}} \, \left(\frac{u}{\|u\|}\right)^{2q} \, dx\leq -\infty.
\end{align*}
Hence there exists $\bar{t}>0$ large enough for which $\varphi_{u}(\bar{t}\,)<0$. By virtue of the continuity of $\varphi_{u}$ and using $(f_{4})$, there exists $t_{u}>0$ which is a global maximum of $\varphi_{u}$ with $t_{u}u \in \mathcal{N}$. 

Next, we aim to prove that such $t_{u}$ is the unique critical point of $\varphi_{u}$. Assume by contradiction that there exist $0<t_{1}<t_{2}$ which are critical points of $\varphi_{u}$. Then, from the definition of $\varphi_{u}$ we get
\begin{align*}
&\frac{1}{t_{1}^{2q-p}}\int_{\R^{3}}\left(a|\nabla u|^{p}+V(x)|u|^{p}\right)\, dx +\frac{b}{t_{1}^{2(q-p)}}\left(\int_{\R^{3}} |\nabla u|^{p}\, dx\right)^{2}+\frac{1}{t_{1}^{q}}\int_{\R^{3}}\left(c|\nabla u|^{q}+V(x)|u|^{q}\right)\, dx \\
&\quad+d\left(\int_{\R^{3}} |\nabla u|^{q}\, dx\right)^{2} -\int_{\R^{3}} K(x)\frac{f(t_{1}u)}{(t_{1}u)^{2q-1}} u^{2q}\, dx=0
\end{align*}
and 
\begin{align*}
&\frac{1}{t_{2}^{2q-p}}\int_{\R^{3}}\left(a|\nabla u|^{p}+V(x)|u|^{p}\right)\, dx +\frac{b}{t_{2}^{2(q-p)}}\left(\int_{\R^{3}} |\nabla u|^{p}\, dx\right)^{2}+\frac{1}{t_{2}^{q}}\int_{\R^{3}}\left(c|\nabla u|^{q}+V(x)|u|^{q}\right)\, dx \\
&\quad+d\left(\int_{\R^{3}} |\nabla u|^{q}\, dx\right)^{2} -\int_{\R^{3}} K(x)\frac{f(t_{2}u)}{(t_{2}u)^{2q-1}} u^{2q}\, dx=0.
\end{align*}
These equalities together with assumption $(f_{4})$ imply that
\begin{align*}
0>&\left(\frac{1}{t_{2}^{2q-p}}- \frac{1}{t_{1}^{2q-p}} \right) \int_{\R^{3}} \left(a|\nabla u|^{p}+V(x)|u|^{p}\right)\, dx + b \left(\frac{1}{t_{2}^{2(q-p)}}-\frac{1}{t_{1}^{2(q-p)}} \right)\left(\int_{\R^{3}}|\nabla u|^{p}\, dx \right)^{2} \\
&\quad + \left(\frac{1}{t_{2}^{q}}- \frac{1}{t_{1}^{q}}\right) \int_{\R^{3}} \left(c|\nabla u|^{q}+V(x)|u|^{q}\right)\, dx\\
&\qquad= \int_{\R^{3}} K(x) \left( \frac{f(t_{2}u)}{(t_{2}u)^{2q-1}}- \frac{f(t_{1}u)}{(t_{1}u)^{2q-1}}\right) u^{2q} \, dx\geq 0
\end{align*}
which leads to a contradiction. 

\noindent
$(b)$ By 
$(a)$ there exists $t_{u}>0$ such that $\varphi'_{u}(t_{u})=0$, or equivalently $\langle \I'(t_{u}u), t_{u}u\rangle =0$, and arguing as before, we find a positive $\tau$ independent of $u$ such that $t_{u}\geq \tau$. 

Now, let $\mathbb{K}\subset \mathbb{S}$ be a compact set and assume by contradiction that there exists $(u_{n})\subset \mathbb{K}$ such that $t_{u_{n}}\ri \infty$. Hence, there exists $u\in \mathbb{K}$ such that $u_{n}\ri u$ in $\E$. Proceeding as in $(a)$ we can prove that $\I(t_{u_{n}}u_{n})\ri -\infty$ in $\R$. Since $t_{u_{n}}u_{n} \in \mathcal{N}$, from Remark \ref{rem1} and recalling that $1<p<q$ we get
\begin{align*}
\I(t_{u_{n}}u_{n}) &= \I(t_{u_{n}}u_{n}) - \frac{1}{2q} \langle \I'(t_{u_{n}}u_{n}), t_{u_{n}}u_{n}\rangle \\
&= \left( \frac{1}{p}-\frac{1}{2q}\right) t_{u_{n}}^{p} \int_{\R^{3}} \left(a|\nabla u_{n}|^{p}+V(x)|u_{n}|^{p}\right)\, dx  + \frac{b}{2}\left(\frac{1}{p}-\frac{1}{q}\right) t_{u_{n}}^{2p}\left(\int_{\R^{3}} |\nabla (u_{n})|^{p}\, dx \right)^{2}\\
&\quad + \frac{1}{2q}t_{u_{n}}^{q} \int_{\R^{3}} \left(c|\nabla u_{n}|^{q}+V(x)|u_{n}|^{q}\right)\, dx + \int_{\R^{3}} K(x)\left[ \frac{1}{2q} f(t_{u_{n}}u_{n}) t_{u_{n}}u_{n} - F(t_{u_{n}}u_{n})\right]\, dx \geq 0, 
\end{align*}
which leads to a contradiction. 

\noindent
$(c)$ Note that $\hat{m}$, $m$ and $m^{-1}$ are well defined. In fact, from $(a)$ we deduce that for each $u\in \E\setminus \{0\}$ there exists a unique $\hat{m}(u)\in \mathcal{N}$. 

On the other hand, if $u\in\mathcal{N}$ then $u\neq 0$, and we deduce that $m^{-1}(u)= \frac{u}{\|u\|} \in \mathbb{S}$ and $m^{-1}$ is well defined. We point out that 
\begin{align*}
&m^{-1}(m(u))= m^{-1}(t_{u}u)=\frac{t_{u}u}{\|t_{u}u\|}=u \quad \mbox{ for any } u\in \mathbb{S}, \\
&m(m^{-1}(u))= m\left(\frac{u}{\|u\|}\right)= t_{\frac{u}{\|u\|}}\, \frac{u}{\|u\|}= u \quad \mbox{ for any } u\in \mathcal{N}, 
\end{align*}
so $m$ is bijective with its inverse $m^{-1}$ continuous. 

Now, let $(u_{n})\subset \E$ and $u\in \E \setminus \{0\}$ such that $u_{n}\ri u$ in $\E$. Using $(b)$ we can find $t_{0}>0$ such that $t_{u_{n}} \|u_{n}\|= t_{\frac{u}{\|u\|}}\ri t_{0}$. Therefore $t_{u_{n}}\ri \frac{t_{0}}{\|u\|}$. Using the fact that $t_{u_{n}}u_{n}\in \mathcal{N}$ and taking the limit as $n\ri \infty$ we deduce that $\frac{t_{0}}{\|u\|}u\in \mathcal{N}$ and $t_{u}= \frac{t_{0}}{\|u\|}$. This implies that $\hat{m}(u_{n})\ri \hat{m}(u)$, hence $\hat{m}$ and $m$ are continuous functions. 
\end{proof}

Let us define the maps
\begin{equation*}
\hat{\psi}: \E \rightarrow \R \quad \mbox{ and } \quad \psi: \mathbb{S}\rightarrow \R,
\end{equation*}
by $\hat{\psi}(u):= \I(\hat{m}(u))$ and $\psi:=\hat{\psi}|_{\mathbb{S}}$. 

The next result is a consequence of Lemma \ref{lem1} (see \cite{SW}).
\begin{prop}\label{propz2}
Suppose that $(V, K)\in \mathcal{K}$ and $f$ fulfills $(f_1)-(f_4)$. Then the following properties hold:
\begin{compactenum}
\item[$(a)$] $\hat{\psi}\in \C^{1}(\E\setminus\{0\}, \R)$ and
\begin{equation*}
\langle \hat{\psi}'(u), v \rangle=\frac{\|\hat{m}(u)\|}{\|u\|} \langle \I'(\hat{m}(u)), v \rangle \, \mbox{ for all } u\in \E\setminus \{0\} \mbox{ and } v\in \E; 
\end{equation*}
\item[$(b)$] $\psi \in \C^{1}(\mathbb{S}, \R)$ and $\langle \psi'(u), v \rangle = \|m(u)\| \langle \I'(m(u)), v \rangle$, for every $v\in T_{u}\mathbb{S}$;
\item[$(c)$] If $(u_{n})$ is a ${(\rm PS)}_{d}$ sequence for $\psi$, then $\{m(u_{n})\}_{n\in \mathbb{N}}$ is a ${(\rm PS)}_{d}$ sequence for $\I$. Moreover, if $(u_{n})\subset \mathcal{N}$ is a bounded ${(\rm PS)}_{d}$ sequence for $\I$, then $\{m^{-1} (u_{n})\}_{n\in \mathbb{N}}$ is a ${(\rm PS)}_{d}$ sequence for the functional $\psi$; 
\item[$(d)$] $u$ is a critical point of $\psi$ if and only if $m(u)$ is a nontrivial critical point for $\I$. Moreover, the corresponding critical values coincide and
\begin{equation*}
\inf_{u\in\mathbb{S}} \psi(u) = \inf_{u\in\mathcal{N}} \I(u).
\end{equation*}
\end{compactenum}
\end{prop}

We notice that the following equalities hold:
\begin{align}\begin{split}\label{z8}
d_{\infty}:= \inf_{u\in \mathcal{N}} \I(u) = \inf_{u\in \E\setminus \{0\}} \max_{t>0} \I(tu) = \inf_{u\in \mathbb{S}} \max_{t>0} \I(tu).
\end{split}\end{align}
In particular, from $(a)$ of Lemma \ref{lem1} and (\ref{z8}) it follows that
\begin{equation}\label{z9}
d_{\infty}>0.
\end{equation}

\section{Technical Lemmas}\label{sect4}

\noindent
For each $u\in \E$ with $u^{\pm}\neq 0$, let us introduce the map $g_{u}: [0, \infty)\times [0, \infty) \ri \R$ defined by 
$$
g_{u}(\xi, \la)= \I(\xi u^{+}+\la u^{-}).
$$ 

\begin{lem}\label{lem3.2}
Suppose that $(V, K)\in \mathcal{K}$ and $f$ fulfills $(f_{1})$-$(f_{4})$. Then the following properties hold:
\begin{compactenum}[$(i)$]
\item The pair $(\xi, \la)$ is a critical point of $g_{u}$ with $\xi, \la>0$ if and only if $\xi u^{+}+ \la u^{-} \in \mathcal{M}$. 
\item The map $g_{u}$ has a unique critical point $(\xi_{+}, \la_{-})$, with $\xi_{+}= \xi_{+}(u)>0$ and $\la_{-}= \la_{-}(u)>0$ which is the unique global maximum point of $g_{u}$. 
\item The maps $a_{+}(r):=\frac{\partial g_{u}}{\partial \xi}(r, \la_{-})$ and $a_{-}(r):=\frac{\partial g_{u}}{\partial \la}(\xi_{+}, r)r$ are such that $a_{+}(r)>0$ if $r\in (0, \xi_{+})$, $a_{+}(r)<0$ if $r\in (\xi_{+}, \infty)$, $a_{-}(r)>0$ if $r\in (0, \la_{-})$ and $a_{-}(r)<0$ if $r\in (\la_{-}, \infty)$.
\end{compactenum}
\end{lem}

\begin{proof}

\noindent
$(i)$
Let us point out that the gradient of $g_{u}$ is given by 
\begin{align*}
\nabla g_{u}(\xi, \la)&= \left( \frac{\partial g_{u}}{\partial \xi} (\xi, \la), \frac{\partial g_{u}}{\partial \la} (\xi, \la) \right) \\
&= \left( \langle \I'(\xi u^{+}+ \la u^{-}), u^{+} \rangle, \langle \I'(\xi u^{+}+ \la u^{-}), u^{-} \rangle \right) \\
&= \left(\frac{1}{\xi} \langle \I'(\xi u^{+}+ \la u^{-}), \xi u^{+} \rangle, \frac{1}{\la}\langle \I'(\xi u^{+}+ \la u^{-}), \la u^{-} \rangle \right). 
\end{align*}
Now, the pair $(\xi, \la)$, with $\xi, \la>0$, is a critical point of $g_{u}$ if and only if 
\begin{align*}
\langle \I'(\xi u^{+}+ \la u^{-}), \xi u^{+} \rangle=0 \quad \mbox{ and } \quad \langle \I'(\xi u^{+}+ \la u^{-}), \la u^{-} \rangle=0
\end{align*}
that is $\xi u^{+}+ \la u^{-}\in \mathcal{M}$.

\noindent
$(ii)$
First we prove that $\mathcal{M}\neq \emptyset$. For each $u\in \E$ with $u^{\pm} \neq 0$ and $\la_{0}$ fixed, let us define the function $g_{1}(\xi): [0, \infty) \ri [0, \infty)$ by $g_{1}(\xi)= g_{u}(\xi, \la_{0})$. 

As in Lemma \ref{lem1}, the map $g_{1}$ has a maximum positive point and furthermore there exists $\xi_{0}= \xi_{0}(u, \la_{0})>0$ such that $g_{1}'(\xi)>0$ for $\xi\in (0, \xi_{0})$, $g_{1}'(\xi)<0$ for $\xi\in (\xi_{0}, \infty)$ and $g_{1}'(\xi_{0})=0$.
 
Hence, it is well defined the function $\eta_{1}:[0, \infty) \ri [0, \infty)$ defined by $\eta_{1}(\la):= \xi(u, \la)$, where $\xi(u, \la)$ satisfies the properties just mentioned with $\la$ in place of $\la_{0}$. Exploiting the definition of $g_{1}$, for all $\la\geq 0$ we get
\begin{equation}\label{t1}
g_{1}'(\eta_{1}(\la))= \frac{\partial g_{u}}{\partial \xi}(\eta_{1}(\la), \la)= \langle \I'(\eta_{1}(\la)u^{+}+ \la u^{-}), \eta_{1}(\la)u^{+}\rangle =0. 
\end{equation}
Note that, when $u^{\pm}\neq 0$ and the support of $u^{+}$ and $u^{-}$ are disjoint in $\R^{3}$, it follows that \eqref{t1} is equivalent to
\begin{align}\begin{split}\label{t2}
&\eta_{1}(\la)^{p}\int_{\R^{3}} \left(a |\nabla u^{+}|^{p}+V(x)|u^{+}|^{p}\right)\, dx+ b\eta_{1}(\la)^{2p} \left(\int_{\R^{3}}|\nabla u^{+}|^{p}\, dx \right)^{2} \\
&\quad + \eta_{1}(\la)^{q} \int_{\R^{3}} \left(c |\nabla u^{+}|^{q}+V(x)|u^{+}|^{q}\right)\, dx + d\eta_{1}(\la)^{2q} \left(\int_{\R^{3}}|\nabla u^{+}|^{q}\, dx \right)^{2} \\
&\qquad= \int_{\R^{3}} K(x) f(\eta_{1}(\la)u^{+})\, \eta_{1}(\la)u^{+}\, dx. 
\end{split}\end{align}
First we note that $\eta_{1}$ is a continuous map. Indeed, let $(\la_{n})$ be a sequence such that $\la_{n}\ri \la_{0}$ as $n\ri \infty$ in $\R$, and assume that $\eta_{1}(\la_{n})\ri \infty$ as $n\ri \infty$. We aim to prove that $(\eta_{1}(\la_{n}))$ is bounded. By contradiction, let us suppose that there us a subsequence, still denoted by $(\la_{n})$, such that $\eta_{1}(\la_{n})\ri \infty$ as $n\ri \infty$. In particular, for $n$ sufficiently large we have that $\eta_{1}(\la_{n})\geq \la_{n}$. From \eqref{t2} we get
\begin{align*}
&\frac{1}{\eta_{1}(\la_{n})^{2q-p}}\int_{\R^{3}} \left(a |\nabla u^{+}|^{p}+V(x)|u^{+}|^{p}\right)\, dx+ \frac{b}{\eta_{1}(\la_{n})^{2(q-p)}} \left(\int_{\R^{3}}|\nabla u^{+}|^{p}\, dx \right)^{2} \\
&\quad+\frac{1}{\eta_{1}(\la_{n})^{q}}  \int_{\R^{3}} \left(c |\nabla u^{+}|^{q}+V(x)|u^{+}|^{q}\right)\, dx+ d \left(\int_{\R^{3}}|\nabla u^{+}|^{q}\, dx \right)^{2} \\
&\qquad= \int_{\R^{3}} K(x) \frac{f(\eta_{1}(\la_{n}) u^{+})}{(\eta_{1}(\la_{n})u^{+})^{2q-1}} (u^{+})^{2q}\, dx, 
\end{align*}
recalling that $\eta_{1}(\la_{n})\ri \infty$ as $n\ri \infty$, $\la_{n}\ri \la_{0}$ as $n\ri \infty$ and exploiting $(f_{3})$, $(f_{4})$ and Fatou's lemma, we get a contradiction. This shows that $(\eta_{1}(\la_{n}))$ is bounded. So there exists $\xi_{0}\geq 0$ such that $\eta_{1}(\la_{n})\ri \xi_{0}$ as $n\ri \infty$. Now, using \eqref{t2} with $\la= \la_{n}$ and taking $n\ri \infty$ we deduce
\begin{align*}
&\xi_{0}^{p}\int_{\R^{3}} \left(a |\nabla u^{+}|^{p}+V(x)|u^{+}|^{p}\right)\, dx+ b\xi_{0}^{2p} \left(\int_{\R^{3}}|\nabla u^{+}|^{p}\, dx \right)^{2} \\
&\quad + \xi_{0}^{q} \int_{\R^{3}} \left(c |\nabla u^{+}|^{q}+V(x)|u^{+}|^{q}\right)\, dx + d\xi_{0}^{2q} \left(\int_{\R^{3}}|\nabla u^{+}|^{q}\, dx \right)^{2} \\
&\qquad= \int_{\R^{3}} K(x) f(\xi_{0}u^{+})\, \xi_{0}u^{+}\, dx
\end{align*}
that is
$$
g_{1}'(\xi_{0})=\frac{\partial g_{u}}{\partial \xi}(\xi_{0}, \la_{0})=0. 
$$
Hence, $\xi_{0}=\eta_{1}(\la_{0})$ which implies that $\eta_{1}$ is a continuous map. 

Moreover, $\eta_{1}(0)>0$. Indeed, if we suppose by contradiction that there exists a sequence $(\la_{n})$ such that $\eta_{1}(\la_{n})\ri 0^{+}$ and $\la_{n}\ri 0$ as $n\ri \infty$, then gathering \eqref{t2} with $(f_{1})$ we get
\begin{align*}
d\left(\int_{\R^{3}} |\nabla u^{+}|^{q}\, dx \right)^{2} \leq \int_{\R^{3}} K(x) \frac{f(\eta_{1}(\la_{n})u^{+})}{(\eta_{1}(\la_{n})u^{+})^{2q-1}} (u^{+})^{2q}\, dx \ri 0, 
\end{align*}
which gives a contradiction. Finally, we can also see that $\eta_{1}(\la)\leq s$ for $s$ sufficiently  large. 

In a similar fashion, for each $\xi_{0}\geq 0$ we define $g_{2}(\la)= g_{u}(\xi_{0}, \la)$, and we can introduce a map $\eta_{2}$ that satisfies the same properties as $\eta_{1}$. In particular, there exists a positive constant $A_{1}$ such that for each $\xi, \la \geq A_{1}$ it holds that $\eta_{1}(\la)\leq \la$ and $\eta_{2}(\xi)\leq \xi$. 

Let $A_{2}:=\max\{\max_{\la\in [0, A_{1}]} \eta_{1}(\la), \max_{\xi\in [0, A_{1}]} \eta_{2}(\xi)\}$ and set $A:= \max\{C_{1}, C_{2}\}$. 

Next, we introduce the map $\Phi: [0, A]\times [0, A]\ri \R^{2}$ defined by $\Phi(\xi, \la)= (\eta_{1}(\la), \eta_{2}(\xi))$. First we note that $\Phi$ is a continuous map due to the continuity of $\eta_{1}$ and $\eta_{2}$, moreover for every $s\in [0, A]$ we can see that 
\begin{align*}
&\mbox{ if } \la\geq A_{1} \mbox{ then } \eta_{1}(\la) \leq \la \leq A \\ 
&\mbox{ if } \la\leq A_{1} \mbox{ then } \eta_{1}(\la)\leq \max_{\la\in [0, A_{1}]} \eta_{1}(\la)\leq A_{2}\leq A
\end{align*}
and similarly
\begin{align*}
&\mbox{ if } \xi\geq A_{1} \mbox{ then } \eta_{2}(\xi)\leq t\leq A, \\
&\mbox{ if } \xi\leq A_{1} \mbox{ then } \eta_{2}(\xi)\leq \max_{\xi\in [0, A_{1}]} \eta_{1}(\xi)\leq A_{2}\leq A, 
\end{align*}
hence $\Phi([0, A]\times [0, A])\subset [0, A]\times [0, A]$. 
Applying Brouwer's fixed point theorem there exists $(\xi_{+}, \la_{-})\in [0, A]\times [0, A]$ such that 
\begin{align*}
\Phi(\xi_{+}, \la_{-})= (\eta_{1}(s_{-}), \eta_{2}(t_{+}))= (\xi_{+}, \la_{-}). 
\end{align*}
Since $\eta_{1}, \eta_{2}$ are positive functions, $\xi_{+}, \la_{-}>0$. In addition $\nabla g_{u}(\xi_{+}, \la_{-})=0$, hence $(\xi_{+}, \la_{-})$ is a critical point of $g_{u}$. 

Next, we show the uniqueness of $(\xi_{+}, \la_{-})$. First, take $w\in \mathcal{M}$. By $w= w^{+}+w^{-}$ and the definition of $g_{w}$ it follows that $\nabla g_{w}(1, 1)=(0, 0)$, hence $(1, 1)$ is a critical point of $g_{w}$. Our aim is to show that $(1, 1)$ is the unique critical point of $g_{w}$ with positive coordinates. With this goal, let $(\xi_{0}, \la_{0})$ be a critical point go $g_{w}$ with $0<\xi_{0}\leq \la_{0}$. Using $\frac{\partial g_{w}}{\partial \xi}(\xi_{0}, \la_{0})=0$, which is equivalent to $\langle \I'(\xi_{0}w^{+}+ \la_{0}w^{-}), \xi_{0}w^{+}\rangle=0$, we can see that
\begin{align*}
&\frac{1}{\xi_{0}^{2q-p}} \int_{\R^{3}} \left( a|\nabla w^{+}|^{p}+ V(x)|w^{+}|^{p}\right) \, dx + \frac{b}{\xi_{0}^{2(q-p)}} \left(\int_{\R^{3}}|\nabla w^{+}|^{p}\, dx\right)^{2} \\
&\quad + \frac{1}{\xi_{0}^{q}} \int_{\R^{3}} \left(c|\nabla w^{+}|^{q} + V(x)|w^{+}|^{q}\right)\, dx + d\left(\int_{\R^{3}}|\nabla w^{+}|^{q}\, dx\right)^{2} \\
&\qquad =\int_{\R^{3}} K(x) \frac{f(\xi_{0} w^{+})}{(\xi_{0}w^{+})^{2q-1}} \, (w^{+})^{2q}\, dx.   
\end{align*}
Exploiting the fact that $w\in \mathcal{M}$ we have 
\begin{align*}
&\int_{\R^{3}} \left( a|\nabla w^{+}|^{p}+ V(x)|w^{+}|^{p}\right) \, dx + b \left(\int_{\R^{3}}|\nabla w^{+}|^{p}\, dx\right)^{2} \\
&\quad + \int_{\R^{3}} \left(c|\nabla w^{+}|^{q} + V(x)|w^{+}|^{q}\right)\, dx + d\left(\int_{\R^{3}}|\nabla w^{+}|^{q}\, dx\right)^{2} \\
&\qquad =\int_{\R^{3}} K(x) f(w^{+}) w^{+}\, dx
\end{align*}
and subtracting we have
\begin{align}\begin{split}\label{t3}
&\left(\frac{1}{\xi_{0}^{2q-p}}-1\right)\int_{\R^{3}} \left( a|\nabla w^{+}|^{p}+ V(x)|w^{+}|^{p}\right) \, dx + b \left(\frac{1}{\xi_{0}^{2q-2p}}-1\right) \left(\int_{\R^{3}}|\nabla w^{+}|^{p}\, dx\right)^{2} \\
&\quad + \left(\frac{1}{\xi_{0}^{q}}-1\right) \int_{\R^{3}} \left(c|\nabla w^{+}|^{q} + V(x)|w^{+}|^{q}\right)\, dx\\
&\qquad = \int_{\R^{3}} K(x) \left(\frac{f(\xi_{0}w^{+})}{(\xi_{0}w^{+})^{2q-1}}- \frac{f(w^{+})}{(w^{+})^{2q}}\right) (w^{+})^{2q}\, dx. 
\end{split}\end{align}
Using \eqref{t3} and $(f_{4})$ we get $\xi_{0}\geq 1$. 

Similarly, from $\frac{\partial g_{w}}{\partial \la}(\xi_{0}, \la_{0})=0$ we obtain
\begin{align*}
&\frac{1}{\la_{0}^{2q-p}} \int_{\R^{3}} \left( a|\nabla w^{-}|^{p}+ V(x)|w^{-}|^{p}\right) \, dx + \frac{b}{\la_{0}^{2(q-p)}} \left(\int_{\R^{3}}|\nabla w^{-}|^{p}\, dx\right)^{2} \\
&\quad + \frac{1}{\la_{0}^{q}} \int_{\R^{3}} \left(c|\nabla w^{-}|^{q} + V(x)|w^{-}|^{q}\right)\, dx + d\left(\int_{\R^{3}}|\nabla w^{-}|^{q}\, dx\right)^{2} \\
&\qquad =\int_{\R^{3}} K(x) \frac{f(\la_{0} w^{-})}{(\la_{0}w^{-})^{2q-1}} \, (w^{-})^{2q}\, dx.   
\end{align*}
Note that from $w\in \mathcal{M}$ we also deduce that 
\begin{align*}
&\int_{\R^{3}} \left( a|\nabla w^{-}|^{p}+ V(x)|w^{-}|^{p}\right) \, dx + b \left(\int_{\R^{3}}|\nabla w^{-}|^{p}\, dx\right)^{2} \\
&\quad + \int_{\R^{3}} \left(c|\nabla w^{-}|^{q} + V(x)|w^{-}|^{q}\right)\, dx + d\left(\int_{\R^{3}}|\nabla w^{-}|^{q}\, dx\right)^{2} \\
&\qquad =\int_{\R^{3}} K(x) f(w^{-}) w^{-}\, dx.   
\end{align*}
Subtracting these last two equality and using assumption $(f_{4})$ we get $0<\xi_{0}\leq \la_{0}\leq 1$.
Hence $\xi_{0}= \la_{0}=1$, this shows that $(1, 1)$ is the unique critical point of $g_{w}$ with positive coordinates. 

Next, take $u\in \E$ such that $u^{\pm}\neq 0$. Let $(\xi_{1}, \la_{1})$ and $(\xi_{2}, \la_{2})$ be two critical points of $g_{u}$ such that $\xi_{i}, \la_{i}>0$ for $i=1, 2$. 
Define
\begin{align*}
U_{1}=\xi_{1}u^{+}+\la_{1}u^{-} \quad \mbox{ and } \quad U_{2}= \xi_{2}u^{+}+ \la_{2}u^{-}. 
\end{align*}
Then we have that $U_{1}, U_{2}\in \mathcal{M}$ and $U_{1}^{\pm}\neq 0$. Furthermore, recalling that $\xi_{1}, \la_{1}>0$ we have
\begin{align*}
\frac{\xi_{2}}{\xi_{1}} U_{1}^{+} + \frac{\la_{2}}{\la_{1}} U_{1}^{-}= \frac{\xi_{2}}{\xi_{1}}\xi_{1}u^{+} + \frac{\la_{2}}{\la_{1}}\la_{1}u^{-}= \xi_{2}u^{+}+ \la_{2}u^{-}= U_{2} \in \mathcal{M}
\end{align*}
hence from $(i)$ we deduce that $\left( \frac{\xi_{2}}{\xi_{1}}, \frac{\la_{2}}{\la_{1}}\right)$ is a critical point of $g_{U_{1}}$. 
Due to the fact that $U_{1}\in \mathcal{M}$ we infer that $\frac{\xi_{2}}{\xi_{1}}= \frac{\la_{2}}{\la_{1}}=1$, that is $\xi_{1}= \xi_{2}$ and $\la_{1}= \la_{2}$, from which follows the uniqueness. 

Now we prove that $g_{u}$ has a maximum global point. Let $\Omega^{+}\subset \supp u^{+}$ and $\Omega^{-}\subset \supp u^{-}$ be positive with finite measure. Gathering \eqref{somma} with $(f_{3})$ and Fatou's lemma we get
\begin{align*}
g_{u}(\xi, \la)&= \I(\xi u^{+})+ \I(\la u^{-})\\
&=\left\{ \xi^{p} \int_{\R^{3}} \left(a|\nabla u^{+}|^{p}+ V(x)|u^{+}|^{p}\right)\, dx+ \frac{b}{2p}\xi^{2p}\left(\int_{\R^{3}} |\nabla u^{+}|^{p}\, dx \right)^{2} \right.\\
&\quad \left.+ \xi^{q} \int_{\R^{3}} \left(c|\nabla u^{+}|^{q}+ V(x)|u^{+}|^{q}\right)\, dx + \frac{d}{2q}\xi^{2q}\left(\int_{\R^{3}} |\nabla u^{+}|^{q}\, dx \right)^{2} -\int_{\Omega^{+}} K(x) F(\xi u^{+})\, dx \right\}\\
&+ \left\{ \la^{p}\int_{\R^{3}} \left(a|\nabla u^{-}|^{p}+ V(x)|u^{-}|^{p}\right)\, dx+ \frac{b}{2p}\la^{2p}\left(\int_{\R^{3}} |\nabla u^{-}|^{p}\, dx \right)^{2} \right. \\
&\quad \left.+ \la^{q} \int_{\R^{3}} \left(c|\nabla u^{-}|^{q}+ V(x)|u^{-}|^{q}\right)\, dx  + \frac{d}{2q}\la^{2q} \left(\int_{\R^{3}} |\nabla u^{-}|^{q}\, dx \right)^{2} -\int_{\Omega^{-}} K(x) F(\la u^{-})\, dx \right\}\\
&\leq \left\{ \xi^{p} \int_{\R^{3}} \left(a|\nabla u^{+}|^{p}+ V(x)|u^{+}|^{p}\right)\, dx+ \frac{b}{2p}\xi^{2p}\left(\int_{\R^{3}} |\nabla u^{+}|^{p}\, dx \right)^{2} \right.\\
&\quad \left.+ \xi^{q} \int_{\R^{3}} \left(c|\nabla u^{+}|^{q}+ V(x)|u^{+}|^{q}\right)\, dx + \frac{d}{2q}\xi^{2q}\left(\int_{\R^{3}} |\nabla u^{+}|^{q}\, dx \right)^{2} -C_{1} \xi^{2q}\int_{\Omega^{+}} K(x) (u^{+})^{2q}\, dx \right\}\\
&+ \left\{ \la^{p}\int_{\R^{3}} \left(a|\nabla u^{-}|^{p}+ V(x)|u^{-}|^{p}\right)\, dx+ \frac{b}{2p}\la^{2p}\left(\int_{\R^{3}} |\nabla u^{-}|^{p}\, dx \right)^{2} \right. \\
&\quad \left.+ \la^{q} \int_{\R^{3}} \left(c|\nabla u^{-}|^{q}+ V(x)|u^{-}|^{q}\right)\, dx  + \frac{d}{2q}\la^{2q} \left(\int_{\R^{3}} |\nabla u^{-}|^{q}\, dx \right)^{2} -C_{1}\la^{2q}\int_{\Omega^{-}} K(x) (u^{-})^{2q}\, dx \right\}\\
&+ C_{2}|K|_{\infty} (|\Omega^{+}|+ |\Omega^{-}|)\ri -\infty \quad \mbox{ as } |(\xi, \la)|\ri \infty.
\end{align*}
Combining the fact that $g_{u}$ is a continuous function with $g_{u}(\xi, \la)\ri -\infty$ as $|(\xi, \la)|\ri \infty$, we conclude that $g_{u}$ assumes a global maximum in $(\bar{\xi}, \bar{\la})\in (0, \infty)\times (0, \infty)$. Using \eqref{somma}, for any $\xi, \la \geq 0$ we get
\begin{align*}
g_{u}(\xi, 0)+ g_{u}(0, \la)= \I(\xi u^{+}) + \I(\la u^{-})=\I(\xi u^{+}+ \la u^{-})=g_{u}(\xi, \la), 
\end{align*}
therefore
\begin{align*}
0<\max_{\xi\geq 0} g_{u}(\xi, 0) < \max_{\xi, \la \geq 0} g_{u}(\xi, \la)=g_{u}(\bar{\xi}, \bar{\la}) 
\end{align*}
and 
\begin{align*}
0<\max_{\la\geq 0} g_{u}(0, \la) < \max_{\xi, \la\geq 0} g_{u}(\xi, \la)=g_{u}(\bar{\xi}, \bar{\la})
\end{align*}
showing that $(\bar{\xi}, \bar{\la})\in (0, \infty)\times (0, \infty)$. By virtue of the uniqueness of the critical point of $g_{u}$ we have that $(\xi_{+}, \la_{-})= (\bar{\xi}, \bar{\la})$, hence $(\xi_{+}, \la_{-})$ is the unique global maximum of $g_{u}$. 

\noindent
$(iii)$ From Lemma \ref{lemz1}-$(a)$ we get $\frac{\partial g_{u}}{\partial \xi}(r, \la_{-})>0$ if $r\in (0, \xi_{+})$, $\frac{\partial g_{u}}{\partial \xi}(\xi_{+}, \la_{-})=0$ and $\frac{\partial g_{u}}{\partial \xi}(r, \la_{-})>0$ if $r\in (\xi_{+}, \infty)$. Similarly for $a_{-}(r)$.
\end{proof}

Proceeding as in \cite{HMH} we can prove the following lemma. 
\begin{lem}\label{lem3.3}
If $(u_{n})\subset \mathcal{M}$ and $u_{n}\rightharpoonup u$ in $\E$, then $u\in \E$ and $u^{\pm} \neq 0$.
\end{lem}

Now, we denote by 
\begin{align}\label{defncinfty}
c_{\infty}= \inf_{u\in \mathcal{M}} \I(u). 
\end{align}
From $\mathcal{M}\subset \mathcal{N}$ it follows that $c_{\infty}\geq d_{\infty}$>0. 

\section{Proof of Theorem \ref{thm1}}\label{sect5}

Let $(u_{n})\subset \mathcal{M}$ be such that 
\begin{align}\label{t4}
\I(u_{n})\ri c_{\infty} \quad \mbox{ in } \R.
\end{align}
First we show that $(u_{n})$ is bounded in $\E$. Suppose that there exists a subsequence still denoted by $(u_{n})$ such that 
$$
\|u_{n}\|\ri \infty \mbox{ as } n\ri \infty. 
$$
Set $v_{n}= \frac{u_{n}}{\|u_{n}\|}$ for all $n\in \mathbb{N}$. Hence $(v_{n})$ is bounded in $\E$ so by Lemma \ref{lem2} we may assume that 
\begin{align}\begin{split}\label{convergenze}
&v_{n}\rightharpoonup v \mbox{ in } \E, \\
&v_{n}\ri v \mbox{ a.e. in } \R^{3}, \\
&v_{n}\ri v \mbox{ in } L^{r}(\R^{3}) \mbox{ for } r\in (q, q^{*}). 
\end{split}\end{align}
Now, from $u_{n}= \|u_{n}\|v_{n}$ it follows that 
\begin{align*}
\|u_{n}\|v_{n}^{+} + \|u_{n}\|v_{n}^{-}= \|u_{n}\|v_{n}=u_{n}\in \mathcal{M}
\end{align*} 
and by Lemma \ref{lem3.2} we have $\xi_{+}(v_{n})=\la_{-}(v_{n})= \|u_{n}\|$. Recalling that $(\xi_{+}, \la_{-})$ is the unique global maximum point of $g_{v_{n}}$ with positive coordinates, for any $\xi>0$ we infer
\begin{align}\begin{split}\label{t5}
\I(u_{n})&= \I(\|u_{n}\|v_{n}) = \I(\xi_{+}(v_{n}) v_{n}^{+}+ \la_{-}(v_{n}) v_{n}^{-})\\
&=g_{v_{n}}(\xi_{+}(v_{n}), \la_{-}(v_{n})) \geq g_{v_{n}}(\xi, \xi) = \I(\xi v_{n})\\
&= \frac{\xi^{p}}{p} \int_{\R^{3}} \left( a|\nabla v_{n}|^{p}+V(x)|v_{n}|^{p}\right)\, dx + \frac{b}{2p}\xi^{2p}\left(\int_{\R^{3}} |\nabla v_{n}|^{p}\, dx \right)^{2} \\
& + \frac{\xi^{q}}{q} \int_{\R^{3}} \left( c|\nabla v_{n}|^{q}+V(x)|v_{n}|^{q}\right)\, dx + \frac{d}{2q}\xi^{2q}\left(\int_{\R^{3}} |\nabla v_{n}|^{q}\, dx \right)^{2} - \int_{\R^{3}} K(x) F(\xi v_{n})\, dx. 
\end{split}\end{align}
Note that $\|v_{n}\|=1$, hence 
\begin{align*}
\int_{\R^{3}} \left( a|\nabla v_{n}|^{p}+V(x)|v_{n}|^{p}\right)\, dx \leq 1 \mbox{ and } \int_{\R^{3}} \left( c|\nabla v_{n}|^{q}+V(x)|v_{n}|^{q}\right)\, dx \leq 1. 
\end{align*}
Using $1<p<q$, and assuming without loss of generality that $\xi>1$ so that $\xi^{q}>\xi^{p}$, and exploiting the following inequality 
$$
a^{q}+ b^{q}\geq C_{q}(a+b)^{q} \mbox{ for all } a, b\geq 0 \mbox{ and } q>1
$$
from \eqref{t5} we deduce 
\begin{align}\begin{split}\label{t6}
\I(u_{n}) &\geq \frac{\xi^{p}}{p} \int_{\R^{3}} \left( a|\nabla v_{n}|^{p}+V(x)|v_{n}|^{p}\right)\, dx + \frac{\xi^{q}}{q} \int_{\R^{3}} \left( c|\nabla v_{n}|^{q}+V(x)|v_{n}|^{q}\right)\, dx - \int_{\R^{3}} K(x) F(\xi v_{n})\, dx\\
&\geq \frac{\xi^{p}}{q}C_{q} \|v_{n}\|^{q}- \int_{\R^{3}} K(x) F(\xi v_{n})\, dx\\
&= \frac{\xi^{p}}{q}C_{q}- \int_{\R^{3}} K(x) F(\xi v_{n})\, dx.
\end{split}\end{align}

\noindent
Assume by contradiction that $v=0$. From \eqref{convergenze} and Lemma \ref{lemconvergFf} we deduce that for any $\xi>0$ 
\begin{align}\label{t7}
\lim_{n\ri \infty} \int_{\R^{3}} K(x) F(\xi v_{n})\, dx=0. 
\end{align}
Taking the limit in \eqref{t6}, and using \eqref{t4} and \eqref{t7} we have
\begin{align*}
c_{\infty}\geq \frac{\xi^{p}}{q}C_{q} \quad \mbox{ for any } \xi>1, 
\end{align*}
which gives a contradiction. Therefore $v\neq 0$.

On the other hand
\begin{align}\begin{split}\label{t8}
\frac{\I(u_{n})}{\|u_{n}\|^{2q}}&=\frac{1}{p\|u_{n}\|^{2q}} \int_{\R^{3}} \left(a|\nabla u_{n}|^{p}+V(x)|u_{n}|^{p}\right)\, dx +\frac{b}{2p\|u_{n}\|^{2q}} \left(\int_{\R^{3}} |\nabla u_{n}|^{p}\, dx \right)^{2}\\
&\quad + \frac{1}{q\|u_{n}\|^{2q}} \int_{\R^{3}} \left(c|\nabla u_{n}|^{q}+V(x)|u_{n}|^{q}\right)\, dx +\frac{d}{2q\|u_{n}\|^{2q}} \left(\int_{\R^{3}} |\nabla u_{n}|^{q}\, dx \right)^{2}\\
&\quad - \int_{\R^{3}} K(x) \frac{F(u_{n})}{\|u_{n}\|^{2q}}\, dx\\
&\leq \frac{1}{p\|u_{n}\|^{2q-p}} + \frac{b}{2p\|u_{n}\|^{2(q-p)}} +\frac{1}{q\|u_{n}\|^{q}}+ \frac{d}{2q}- \int_{\R^{3}} K(x) \frac{F(\|u_{n}\|v_{n})}{(\|u_{n}\|v_{n})^{2q}} \, \left(\frac{u_{n}}{\|u_{n}\|}\right)^{2q} \, dx. 
\end{split}\end{align}
Combining assumption $(f_{3})$ together with Fatou's lemma we get
\begin{align}\label{t9}
\lim_{n\ri \infty} \int_{\R^{3}} K(x) \frac{F(\|u_{n}\|v_{n})}{(\|u_{n}\|v_{n})^{2q}} \, \left(\frac{u_{n}}{\|u_{n}\|}\right)^{2q} \, dx =+\infty
\end{align}
so taking the limit in \eqref{t8} we get a contradiction in view of \eqref{t4}, $\|u_{n}\|\ri \infty$ as $n\ri \infty$ and \eqref{t9}. So $(u_{n})$ is a bounded sequence in $\E$ and there exists $u\in \E$ such that $u_{n}\rightharpoonup u$ in $\E$. From Lemma \ref{lem3.3} we have $u^{\pm}\neq 0$ and by Lemma \ref{lem3.2} there are $\xi_{+}, \la_{-}>0$ such that $\xi_{+}u^{+}+\la_{-}u^{-}\in \mathcal{M}$, from which 
\begin{align}\begin{split}\label{t10}
&\frac{1}{\la_{-}^{2q-p}} \int_{\R^{3}} \left( a|\nabla u^{-}|^{p}+ V(x)|u^{-}|^{p}\right) \, dx + \frac{b}{\la_{-}^{2(q-p)}} \left(\int_{\R^{3}}|\nabla u^{-}|^{p}\, dx\right)^{2} \\
&\quad + \frac{1}{\la_{-}^{q}} \int_{\R^{3}} \left(c|\nabla u^{-}|^{q} + V(x)|u^{-}|^{q}\right)\, dx + d\left(\int_{\R^{3}}|\nabla u^{-}|^{q}\, dx\right)^{2} \\
&\qquad =\int_{\supp u^{-}} K(x) \frac{f(\la_{-} u^{-})}{(\la_{-}u^{-})^{2q-1}} \, (u^{-})^{2q}\, dx.   
\end{split}\end{align} 

Our aim is to prove that $\xi_{+}=\la_{-}=1$. Without loss of generality, let us suppose that $0<\xi_{+}\leq \la_{-}$. First we prove that $0<\xi_{+}\leq \la_{-}\leq 1$. 
Note that from $u_{n}\rightharpoonup u$ in $\E$ and exploiting Lemma \ref{lemconvergFf} we have 
\begin{align}\label{convf}
\lim_{n\ri \infty} \int_{\R^{3}} K(x) f(u_{n}^{\pm})\, u_{n}^{\pm}\, dx = \int_{\R^{3}} K(x) f(u^{\pm})\, u^{\pm}\, dx 
\end{align} 
and also
\begin{align}\label{convF}
\lim_{n\ri \infty} \int_{\R^{3}} K(x) F(u_{n}^{\pm})\, dx = \int_{\R^{3}} K(x) F(u^{\pm})\, dx, 
\end{align}
and combining $(u_{n})\subset \mathcal{M}$ with Fatou's lemma we get
\begin{align*}
\langle \I'(u), u^{\pm}\rangle \leq \liminf_{n\ri \infty} \langle \I'(u_{n}), u_{n}^{\pm}\rangle =0, 
\end{align*}
which yields
\begin{align}\begin{split}\label{t11}
&\int_{\R^{3}} \left( a|\nabla u^{-}|^{p}+ V(x)|u^{-}|^{p}\right) \, dx + b \left(\int_{\R^{3}}|\nabla u^{-}|^{p}\, dx\right)^{2} \\
&\quad + \int_{\R^{3}} \left(c|\nabla u^{-}|^{q} + V(x)|u^{-}|^{q}\right)\, dx + d\left(\int_{\R^{3}}|\nabla u^{-}|^{q}\, dx\right)^{2} \\
&\qquad \leq \int_{\supp u^{-}} K(x) f(u^{-})\, u^{-}dx.   
\end{split}\end{align}
Subtracting \eqref{t10} and \eqref{t11} we obtain 
\begin{align*}
&\left(\frac{1}{\la_{-}^{2q-p}} -1\right) \int_{\R^{3}} \left( a|\nabla u^{-}|^{p}+ V(x)|u^{-}|^{p}\right) \, dx +b \left(\frac{1}{\la_{-}^{2(q-p)}} -1\right)\left(\int_{\R^{3}}|\nabla u^{-}|^{p}\, dx\right)^{2} \\
&\quad + \left(\frac{1}{\la_{-}^{q}}-1\right) \int_{\R^{3}} \left(c|\nabla u^{-}|^{q} + V(x)|u^{-}|^{q}\right)\, dx \\
&\qquad \geq \int_{\supp u^{-}} K(x) \left(\frac{f(\la_{-} u^{-})}{(\la_{-}u^{-})^{2q-1}}- \frac{u^{-}}{(u^{-})^{2q-1}}\right) \, (u^{-})^{2q}\, dx
\end{align*} 
and using assumption $(f_{3})$ we deduce $0<\la_{-}\leq1$. Hence, $0<\xi_{+}\leq \la_{-}\leq 1$. 

Next we show that 
\begin{align}\label{t12}
\I(\xi_{+}u^{+}+\la_{-}u^{-})= c_{\infty}. 
\end{align}

Now, from \eqref{defncinfty}, $0<\xi_{+}\leq \la_{-}\leq 1$, assumption $(f_{4})$, \eqref{convf} and \eqref{convF} we obtain 

\begin{align*}
c_{\infty}&\leq \I(\xi_{+}u^{+}+\la_{-}u^{-})= \I(\xi_{+}u^{+}+\la_{-}u^{-})- \frac{1}{2q} \langle \I'(\xi_{+}u^{+}+\la_{-}u^{-}), \xi_{+}u^{+}+\la_{-}u^{-}\rangle \\
&=\left(\frac{1}{p}-\frac{1}{2q}\right) \int_{\R^{3}} \left( a|\nabla(\xi_{+}u^{+}+\la_{-}u^{-})|^{p}+ V(x)|\xi_{+}u^{+}+\la_{-}u^{-}|^{p}\right)\, dx \\
&\quad +\frac{b}{2}\left(\frac{1}{p}-\frac{1}{q}\right) \left(\int_{\R^{3}} |\nabla (\xi_{+}u^{+}+\la_{-}u^{-})|^{p}\, dx \right)^{2} \\
&\quad + \frac{1}{2q} \int_{\R^{3}} \left(c|\nabla (\xi_{+}u^{+}+\la_{-}u^{-})|^{q} + V(x)|\xi_{+}u^{+}+\la_{-}u^{-}|^{q}\right)\, dx\\
&\quad+\int_{\R^{3}} K(x) \left(\frac{1}{2q} f(\xi_{+}u^{+}) \,(\xi_{+}u^{+})- F(\xi_{+}u^{+})\right)\, dx +\int_{\R^{3}} K(x) \left(\frac{1}{2q} f(\la_{-}u^{-}) \,(\la_{-}u^{-})- F(\la_{-}u^{-})\right)\, dx \\   
&=\left(\frac{1}{p}-\frac{1}{2q}\right) \int_{\R^{3}} \left( a|\nabla u|^{p}+ V(x)|u|^{p}\right)\, dx +\frac{b}{2}\left(\frac{1}{p}-\frac{1}{q}\right) \left(\int_{\R^{3}} |\nabla u|^{p}\, dx \right)^{2} \\
&\quad + \frac{1}{2q} \int_{\R^{3}} \left(c|\nabla u|^{q} + V(x)|u|^{q}\right)\, dx\\
&\quad+\int_{\R^{3}} K(x) \left(\frac{1}{2q} f(u^{+}) \,u^{+}- F(u^{+})\right)\, dx +\int_{\R^{3}} K(x) \left(\frac{1}{2q} f(u^{-}) \,u^{-}- F(u^{-})\right)\, dx \\   
&=\left(\frac{1}{p}-\frac{1}{2q}\right) \int_{\R^{3}} \left( a|\nabla u|^{p}+ V(x)|u|^{p}\right)\, dx +\frac{b}{2}\left(\frac{1}{p}-\frac{1}{q}\right) \left(\int_{\R^{3}} |\nabla u|^{p}\, dx \right)^{2} \\
&\quad + \frac{1}{2q} \int_{\R^{3}} \left(c|\nabla u|^{q} + V(x)|u|^{q}\right)\, dx +\int_{\R^{3}} K(x) \left(\frac{1}{2q} f(u) \,u- F(u)\right)\, dx \\
&\leq \liminf_{n\ri \infty} \left\{ \left(\frac{1}{p}-\frac{1}{2q}\right) \int_{\R^{3}} \left( a|\nabla u_{n}|^{p}+ V(x)|u_{n}|^{p}\right)\, dx +\frac{b}{2}\left(\frac{1}{p}-\frac{1}{q}\right) \left(\int_{\R^{3}} |\nabla u_{n}|^{p}\, dx \right)^{2} \right.  \\
&\quad \left.+ \frac{1}{2q} \int_{\R^{3}} \left(c|\nabla u_{n}|^{q} + V(x)|u_{n}|^{q}\right)\, dx +\int_{\R^{3}} K(x) \left(\frac{1}{2q} f(u_{n}) \,u- F(u_{n})\right)\, dx\right\} \\
&=\liminf_{n\ri \infty} \left\{ \I(u_{n})-\frac{1}{2q}\langle \I'(u_{n}), u_{n}\rangle \right\}=c_{\infty}
\end{align*}
which implies that \eqref{t12} holds true. In particular it follows that $\xi_{+}=\la_{-}=1$. 

Next, we prove that the minimum point $u=u^{+}+u^{-}$ is a critical point of $\I$. Assume by contradiction that $\I'(u)\neq 0$. Then, due to the continuity of $\I'$ we can find $\alpha, \beta>0$ such that $\| \I'(v)\|\geq \beta$ for all $v\in \E$ with $\|v-u\|\leq 3 \alpha$. 

Define $D= [\frac{1}{2}, \frac{3}{2}]\times [\frac{1}{2}, \frac{3}{2}]$ and $\E^{\pm}= \{u\in \E : u^{\pm}\neq 0\}$, and let us consider the function $G_{u}: D\rightarrow \E^{\pm}$ defined by setting
$$
G_{u}(\xi, \la)= \xi u^{+}+ \la u^{-}.
$$ 
Using Lemma \ref{lem3.2} we can see that $\I(G_{u}(1, 1))=c_{\infty}$ and $\I(G_{u}(\xi, \la))<c_{\infty}$ in $D\setminus\{(1,1)\}$. 

Define 
$$
\tau=\max_{(\xi, \la)\in \partial D} \I(G_{u}(\xi, \la)),
$$ 
then $\tau <c_{\infty}$. 

Set $\tilde{\mathcal{S}}=\{v\in \E : \|v-u\|\leq \alpha \}$ and choose $\e=\min \left\{ \frac{1}{4}(c_{\infty}-\gamma), \frac{\alpha \beta}{8}\right\}$. By Theorem 2.3 in \cite{Willem} there exists a deformation $\eta\in \C([0,1]\times \E, \E)$ such that the following assertions hold:
\smallskip
\begin{compactenum}[(a)]
\item $\eta(\xi,v)= v$ if $v \not\in \E^{-1}([c_{\infty}-2\varepsilon, c_{\infty}+2\varepsilon])$;
\item $\I(\eta(1,v))\leq c_{\infty}-\varepsilon$ for each $v\in \E$ with $\|v-u\|\leq \alpha$ and $\I(v)\leq c_{\infty}+\varepsilon$;
\item $\I(\eta(1,v))\leq \I(\eta(0, v))=\I(v)$ for all $v\in \E$.
\end{compactenum}

From $(b)$ and $(c)$ we get
\begin{equation}\label{c24}
\max_{(\xi, \la)\in \partial D} \I(\eta(1, G_{u}(\xi, \la))) <c_{\infty}.
\end{equation}
Now we prove that 
\begin{equation}\label{c25}
\eta(1, G_{u}(D)) \cap \mathcal{M}\neq \emptyset
\end{equation}
because the definition of $c_{\infty}$ and (\ref{c25}) contradict (\ref{c24}). 

Let us define 
\begin{align*}
&\Phi_{u}(\xi, \la)= \eta(1, G_{u}(\xi, \la)), \\
&\psi_{0}(\xi, \la)= \left(\langle \I'(G_{u}(\xi,1)), \xi u^{+}\rangle, \langle \I'(G_{u}(1,\la)), \la u^{-}\rangle \right) \\
&\psi_{1}(\xi, \la)= \left(\frac{1}{\xi} \langle \I'(\Phi_{u}(\xi,1)), \Phi_{u}^{+}(\xi,1)\rangle, \frac{1}{\la} \langle \I'(\Phi_{u}(1,\la)), \Phi_{u}^{-}(1,\la)\rangle \right).
\end{align*}

Exploiting Lemma \ref{lem3.2}-$(iii)$, the $\C^{1}$ function $\gamma_{+}(\xi)= g_{u}(\xi, 1)$ 
has a unique global maximum point $\xi=1$. 
By density, given $\e>0$ small enough, there is $\gamma_{+, \varepsilon}\in \C^{\infty}([\frac{1}{2}, \frac{3}{2}])$ such that $\|\gamma_{+} - \gamma_{+, \varepsilon}\|_{\C^{1}([\frac{1}{2}, \frac{3}{2}])}<\varepsilon$ with $\xi_{+}$ being the unique maximum global point of $\gamma_{+, \varepsilon}$ in $[\frac{1}{2}, \frac{3}{2}]$. Hence, $\|\gamma'_{+} - \gamma'_{+, \varepsilon}\|_{\C([\frac{1}{2}, \frac{3}{2}])}<\varepsilon$, $\gamma'_{+, \varepsilon}(1)=0$ and $\gamma''_{+, \varepsilon}(1)<0$. Analogously, set $\gamma_{-}(\la)=g_{u}(1,\la)$, then there exists $\gamma_{-, \varepsilon}\in \C^{\infty}([\frac{1}{2}, \frac{3}{2}])$ such that $\|\gamma'_{-} - \gamma'_{-, \varepsilon}\|_{\C([\frac{1}{2}, \frac{3}{2}])}<\varepsilon$, $\gamma'_{+, \varepsilon}(1)=0$ and $\gamma''_{+, \varepsilon}(1)<0$. 

Let us define $\psi_{\varepsilon}\in \C^{\infty}(D)$ by setting 
$$
\psi_{\varepsilon}(\xi, \la)= (\xi \gamma'_{+, \varepsilon}(\xi), \la \gamma'_{-, \varepsilon}(\la)).
$$ 
We note that $\|\psi_{\varepsilon} - \psi_{0} \|_{\C(D)}<\frac{3\sqrt{2}}{2} \varepsilon$, $(0,0) \not \in \psi_{\varepsilon}(\partial D)$, and $(0,0)$ is a regular value of $\psi_{\varepsilon}$ in $D$.

Since $(1,1)$ is the unique solution of $\psi_{\varepsilon}(\xi, \la)=(0,0)$ in $D$, by the definition of Brouwer's degree, we can infer that, for $\varepsilon$ small enough, it holds 
\begin{equation}\label{pv1}
\rm{deg}(\psi_{0}, D, (0,0))= \rm{deg}(\psi_{\varepsilon}, D, (0,0))= \rm{sgn} \, \rm{Jac}(\psi_{\varepsilon})(1,1), 
\end{equation}
where $\rm{Jac}(\psi_{\varepsilon})$ is the Jacobian determinant of $\psi_{\varepsilon}$ and $\rm{sgn}$ denotes the sign function. 

We note that
\begin{align}\label{pv2}
\rm{Jac}(\psi_{\varepsilon})(1,1)= [\gamma'_{+, \varepsilon}(1)+\gamma''_{+, \varepsilon}(1)]\times [\gamma'_{-, \varepsilon}(1)+ \gamma''_{-, \varepsilon}(1)]= \gamma''_{+, \varepsilon}(1)\times \gamma''_{-, \varepsilon}(1)>0, 
\end{align}
so combining \eqref{pv1} with \eqref{pv2} we find
\begin{align*}
\rm{deg}(\psi_{0}, D, (0,0))= \rm{sgn}[\gamma''_{+, \varepsilon}(1)\times \gamma''_{-, \varepsilon}(1)]=1. 
\end{align*}

By the definition of $\tau$ and the fact that $\e=\min \left\{ \frac{1}{4}(c_{\infty}-\gamma), \frac{\alpha \beta}{8}\right\}$ we have that for any $(\xi, \la)\in \partial D$
\begin{align*}
\I(G_{u}(\xi, \la))&\leq \max_{(\xi, \la)\in \partial D} \I(G_{u}(\xi, \la))<\frac{1}{2}(\tau + c_{\infty})=c_{\infty} - 2 \left( \frac{c_{\infty}-\tau}{4} \right)\leq c_{\infty}-2\varepsilon. 
\end{align*}
This and $(a)$ yields that $G_{u}=\Phi_{u}$ on $\partial D$. Therefore, $\psi_{1}=\psi_{0}$ on $\partial D$ and consequently
\begin{equation*}
\rm{deg}(\psi_{1}, D, (0,0))= \rm{deg}(\psi_{0}, D, (0,0))=1,
\end{equation*}
which shows that $\psi_{1}(\xi, \la)=(0,0)$ for some $(\xi, \la)\in D$.

Now, in order  to verify that (\ref{c25}) holds, we prove that
\begin{equation}\label{fig}
\psi_{1}(1, 1)= \left(\langle \I'(\Phi_{u}(\xi,1)), \Phi_{u}(1,1)^{+}\rangle, \langle \I'(\Phi_{u}(1,1)), \Phi_{u}(1,1)^{-}\rangle \right)=0.
\end{equation}
As a matter of fact, (\ref{fig}) and the fact that $(1, 1)\in D$, yield $\Phi_{u}(1, 1)=\eta(1, G_{u}(1, 1))\in \mathcal{M}$.

We argue as follows. If the zero $(\xi, \la)$ of $\psi_{1}$ obtained above is equal to $(1, 1)$ there is nothing to prove. Otherwise, we take $0<\delta_{1}<\min\{|\xi-1|, |\la-1|\}$ and consider 
$$
D_{1}=\left[1-\frac{\delta_{1}}{2}, 1+\frac{\delta_{1}}{2}\right]\times \left[1-\frac{\delta_{1}}{2}, 1+\frac{\delta_{1}}{2}\right].
$$

Therefore $(\xi, \la)\in D \setminus D_{1}$. Hence, we can repeat for $D_{1}$ the same argument used for $D$, so that we can find a couple $(\xi_{1}, \la_{1})\in D_{1}$ such that $\psi_{1}(\xi_{1}, \la_{1}) = 0$. If $(\xi_{1}, \la_{1})=(1, 1)$, there is nothing to prove. Otherwise, we can continue with this procedure and find in the $n$-th step that (\ref{fig}) holds, or produce a sequence $(\xi_{n}, \la_{n})\in D_{n-1}\setminus D_{n}$ which converges to $(1, 1)$ and such that
\begin{equation}\label{italo}
\psi_{1}(\xi_{n}, \la_{n})=0,\quad \mbox{ for every } n\in \mathbb{N}.
\end{equation}
Thus, taking the limit as $n \rightarrow \infty$ in (\ref{italo}) and using the continuity of $\psi_{1}$ we get (\ref{fig}).
Therefore, $u=u^{+}+u^{-}$ is a critical point of $\I$.

Finally, we consider the case when $f$ is odd. Clearly, the functional $\psi$ is even. In the light of (\ref{z9}) and $c_{\infty}\geq d_{\infty}>0$ we can see that $\psi$ is bounded from below in $\mathbb{S}$. Moreover, using Lemma \ref{lem2} and Lemma \ref{lemconvergFf}, we deduce that $\psi$ satisfies the Palais-Smale condition on $\mathbb{S}$. Hence, applying Proposition \ref{propz2} and \cite{Rab}, we conclude that $\I$ has infinitely many critical points.

\section*{Acknowledgments.}
The first author was partly supported by the GNAMPA Project 2020 entitled: {\it Studio Di Problemi Frazionari Nonlocali Tramite Tecniche Variazionali}.
The second author was partly supported by Slovenian research agency grants P1-0292, N1-0114, N1-0083, N1-0064 and J1-8131.

\end{document}